\documentclass{amsart}

\usepackage{amsthm}
\usepackage{amssymb}
\usepackage{amsmath}
\usepackage{amsfonts}
\usepackage{amssymb}
\usepackage{microtype}
\usepackage{BOONDOX-cal}
\usepackage{enumitem}

\newtheorem{thrm}{Theorem}[section]
\newtheorem{lem}[thrm]{Lemma}
\newtheorem{main}{Theorem}

\newcommand{\centerof}[1]{\mathbf{Z}(#1)}
\newcommand{\centralizer}[2]{\mathbf{C}_{#1}(#2)}
\newcommand{\normalizer}[2]{\mathbf{N}_{#1}(#2)}
\newcommand{\id}{\{1\}}    

\newcommand{\syl}[2]{\textnormal{Syl}_{#1}(#2)}
\newcommand{\nsub}{\trianglelefteq}
\newcommand{\characteristic}{\mathrel{\textnormal{char}}}

\begin{document}

\title[$2$-Frobenius group]{The Cyclic Graph of a $2$-Frobenius Group}

\author[D.G.\ Costanzo]{David G.\ Costanzo}
\address{School of Mathematical and Statistical Sciences, Clemson University, Clemson, SC 29634, USA}
\email{dgcosta@clemson.edu}

\author[M.L.\ Lewis]{Mark L.\ Lewis}
\address{Department of Mathematical Sciences, Kent State University, Kent, OH 44242, USA}
\email{lewis@math.kent.edu}


\begin{abstract}
The cyclic graph of a group $G$ is the graph whose vertices are the nonidentity elements of $G$ and whose edges connect distinct elements $x$ and $y$ if and only if the subgroup $\langle x,y\rangle$ is cyclic.
We obtain information about the cyclic graph of $2$-Frobenius groups.  The cyclic graph of a $2$-Frobenius group is disconnected.  In this paper, we determine the number of connected components of the cyclic graph of any $2$-Frobenius group.
\end{abstract}

\maketitle
	
\section{Introduction}

All groups in this paper are finite.  Given a  nonabelian group $G$, the \textit{commuting graph} of $G$ is the graph whose vertex set is $G\setminus\centerof{G}$, the set of non-central elements of $G$, and there is an edge between distinct vertices $x$ and $y$ when $xy=yx$. We denote this graph by $\Gamma(G)$. 
A related graph, denoted by $\Delta(G)$, takes $G^{\#}=G\setminus\id$ as its vertex set, and there is an edge between distinct vertices $x$ and $y$ when $\langle x,y\rangle$ is cyclic. We call this the \textit{cyclic graph} of $G$.
The cyclic graph of a group has also appeared in the literature under name \textit{punctured enhanced power graph}.  

The commuting graph has been studied in many papers; more than we want to cite here.  For our purposes, \cite{parker} is the most relevant.  We believe the cyclic graph was first studied in \cite{imperatore}.  As we have said, it is closely related to the enhanced power graph which was introduced in \cite{cameron} and studied in \cite{bera} among others.  The work in this paper arose in conjunction with the work done on the cyclic graph in \cite{directprod} and \cite{zgroup}.

Our primary focus is the cyclic graphs of $2$-Frobenius groups.   A group $G$ is a \textit{$2$-Frobenius group} if it has normal subgroups $K$ and $L$ such that $L$ is a Frobenius group with Frobenius kernel $K$ and $G/K$ is a Frobenius group with Frobenius kernel $L/K$.  The best known example of a $2$-Frobenius group is $S_{4}$, and in Section \ref{last sec}, we will present a number of other examples of $2$-Frobenius groups.  When $G$ is a $2$-Frobenius group, $\centerof{G}=\id$, and so, $\Delta(G)$ is a spanning subgraph of $\Gamma(G)$.

Parker in \cite{parker} shows that $\Gamma(G)$ is disconnected when $G$ is a Frobenius or a $2$-Frobenius group, and so $\Delta(G)$ is disconnected.  In fact, Parker proves that the only solvable groups $G$ with trivial center that have $\Gamma (G)$ disconnected are Frobenius and $2$-Frobenius groups. 

We demonstrate easily that the number of connected components for $\Gamma(G)$ and $\Delta(G)$ are the same when $G$ is a Frobenius group.  Also, it is not difficult to compute the number of connected components of $\Gamma(G)$ when $G$ is a $2$-Frobenius group.  However, counting the number of connected components of $\Delta (G)$ is more complicated when $G$ is $2$-Frobenius group, and this is the main goal of this paper.

We split this count  into three different cases.  The first case is when $K$ does not have prime power order, the second case is $K$ and $G/L$ are $p$-groups for some prime $p$, and the third case is when $K$ is a $p$-group for some prime $p$ and $G/L$ is not a $p$-group. We will see when $K$ does not have prime power order that the number of connected components for $\Delta(G)$ is the same as for $\Gamma(G)$, but in the other two cases, the count for the number of connected components of $\Delta(G)$ will differ from the count for $\Gamma(G)$.

We first have the formula when $K$ does not have prime power order.

\begin{main} \label{main1}
Let $G$ be a $2$-Frobenius group with $K$ as in the definition.  If $|K|$ is divisible by at least two distinct prime numbers, then $\Delta(G)$ has $|K|+1$ connected components.
\end{main}

In the remaining two cases, $K$ is a group of prime power order.  We see that the counts in these situations are more complicated.  We write $m_{p}(G)$ to denote the number of subgroups of order $p$ of $G$.  Next, we find the formula for the case that $K$ and $G/L$ are $p$-groups for some prime $p$.

\begin{main} \label{main2}
Let $G$ be a $2$-Frobenius group, and assume that $K$ and $G/L$ are $p$-groups for some prime $p$, where $K$ and $L$ are as in the definition.  Then $\Delta(G)$ has $|K|+m_{p}(G)$ connected components.
\end{main}

Finally, we compute the formula when $K$ is a $p$-group and $G/L$ is not a $p$-group for some prime $p$.

\begin{main} \label{main3}
Let $G$ be a $2$-Frobenius group, and let $p$ be a prime number.
Assume that $K$ is a $p$-group for some prime $p$ and that $G/L$ is not a $p$-group, where $K$ and $L$ are as in the definition.  Then the number of connected components of $\Delta(G)$ is $$|K|+|L:K|+m_{p}^{\ast},$$ where $m_{p}^{\ast}$ is the number of subgroups of order $p$ in $G$ that are not centralized by an element of prime order other than $p$.
\end{main}

In the final section, we compute the number of connected components of $\Delta(G)$ for several different $2$-Frobenius groups $G$.

Portions of this paper appear as a part of the first author’s dissertation under the direction of the second author at Kent State University.
Some of this research was conducted during the summer 2019 REU at Kent State University with the funding of NSF Grant DMS-1653002.
We thank the NSF and Profs. Soprunova and Chebotar for their support.
We also thank Stefano Schmidt, Eyob Tsegaye, and Gabe Udell for several discussions.

\section{Structural properties of $2$-Frobenius groups}

We start by fixing some notation.  Let $G$ be a group, and fix elements $x,y\in G$.  When the subgroup $\langle x,y\rangle$ is cyclic, we write $x\approx y$.  Note that $x \approx y$ indicates that either $x$ and $y$ are adjacent in the cyclic graph of $G$ or $x=y$.  Similarly, when $xy=yx$, we write $x\sim y$, and this indicates either $x$ and $y$ are adjacent in the commuting graph or $x=y$.


In the literature, results concerning the structure of a $2$-Frobenius group are often merely mentioned or used implicitly.
As such, we feel that it is appropriate to write a brief section that gathers up some of these structural properties, along with the easy proofs.  These results are known in the folklore and are not original to us.

When $G$ is a $2$-Frobenius group, we use $K$ and $L$ to denote the subgroups that are defined in the definition.  We write $H$ for a Frobenius complement for $L$, we set $N=\normalizer{G}{H}$, and we take to $D$ to denote a subgroup of $G$ such that $D/K$ is a Frobenius complement for $G/K$.  When $G$ is a $2$-Frobenius group, we chose $H$ to be a Frobenius complement for the normal subgroup $L$, but we now show that $H$ is also a Frobenius kernel of its normalizer.

\begin{lem}
If $G$ is a $2$-Frobenius group, then $N$ is a Frobenius group with Frobenius kernel $H$.
\end{lem}

\begin{proof}
By a Frattini Argument, $G=KN$.
We have $G/K\cong N$ as $K\cap N=\normalizer{K}{H}=\id$, and so $N$ is a Frobenius group.
The Frobenius kernel of $N$ has order $|L:K|=|H|$.
Since $H$ is normal in $N$, the Frobenius kernel of $N$ must be $H$ by Theorem 12.6.8 in \cite{wrscott}. 
\end{proof}

It follows that  $H$ is simultaneously a Frobenius kernel and a Frobenius complement.  We now show that its structure is quite limited.

\begin{lem}\label{Hcyclic}
If $G$ is a $2$-Frobenius group, then $H$ is cyclic of odd order.
\end{lem}

\begin{proof}
Seeking a contradiction, we assume $|H|$ is even.
As $H$ is a Frobenius complement, it has a unique involution; call it $z$.
Note that $o(z)=o(z\sigma)$ for every $\sigma\in\textnormal{Aut}(H)$.
Hence $z=z\sigma$, and so $z\in\textnormal{Fix}_{H^{\#}}(\textnormal{Aut}(H))$.
In particular, $H$ does \textit{not} admit a fixed-point-free automorphism.
But $H$ is also a Frobenius kernel.
These conditions are not compatible; thus $|H|$ must be odd.

As $H$ is a Frobenius kernel, $H$ is nilpotent and every Sylow subgroup of $H$ is normal in $H$.
As $H$ is a Frobenius complement of odd order, the Sylow subgroups of $H$ are cyclic.
Thus $H$ is cyclic. 
\end{proof}

The structure of a Frobenius complement for $N$ is also easily determined.

\begin{lem}\label{normalizerDHcyclic}
If $G$ is a $2$-Frobenius group, then $\normalizer{D}{H}$ is cyclic.
\end{lem}

\begin{proof}
Note that $G=HD$, and so $N=N\cap HD=H(N\cap D)=H\normalizer{D}{H}$, by Dedekind's Lemma.
As $H\cap\normalizer{D}{H}=\id$, the subgroup $\normalizer{D}{H}$ is a Frobenius complement for $N$.
Observe that $\centralizer{D}{H}=\id$.
Thus $\normalizer{D}{H}$ embeds in $\textnormal{Aut}(H)$.
By Lemma \ref{Hcyclic}, $\textnormal{Aut}(H)$ is abelian.
Hence $\normalizer{D}{H}$ is an abelian Frobenius complement, which yields the conclusion.
\end{proof}

Next, we obtain a factorization of the subgroup $D$.

\begin{lem}\label{factorizationofD}
If $G$ is a $2$-Frobenius group, then $D=K\normalizer{D}{H}$.
\end{lem}

\begin{proof}
Simply observe that $D=D\cap G=D\cap KN=K(D\cap N)=K\normalizer{D}{H}$ by Dedekind's Lemma.
\end{proof}

It is now easy to see that $D$ is a Hall subgroup of $G$.

\begin{lem}\label{DisHall}
If $G$ is a $2$-Frobenius group, then $D$ is a Hall subgroup of $G$.
\end{lem}

\begin{proof}
Notice that $|G:D|=|H|$.
Using Lemma \ref{factorizationofD}, $|D|=|K||\normalizer{D}{H}|$.
The result follows.
\end{proof}

\section{Preliminaries}

In this section, we include some specific results concerning the cyclic graph of a group.  We begin with two results that appeared in \cite{directprod}.  The first is a basic fact about the cyclic graph of a $p$-group that is used in the proof of Theorem \ref{Dpgroup}.   The proof of this lemma is easy and can be found as Lemma 2.3 of \cite{directprod}.

\begin{lem}\label{cyclicgraphofapgroup}
If $G$ is a $p$-group for some prime number $p$, then $\Delta(G)$ has $m_{p}(G)$ connected components.  
\end{lem}

If $G$ is a group and $x,y\in G^{\#}$, then a sufficient condition for $x$ and $y$ to be adajcent in $\Delta(G)$ is that $x$ and $y$ have coprime orders and commute.  The following lemma is a consequence of this fact.  Again, the proof of this lemma is easy and can be found as Corollary 4.2 in \cite{directprod}.

\begin{lem}\label{nilpotentnotapgroup}
If $G$ is a nilpotent group such that $|G|$ is divisible by at least two distinct primes, then $\Delta(G)$ is connected.
\end{lem}

The following technical lemma will be used in the proof of Theorem \ref{completecount}.

\begin{lem}\label{pgroupcentralizer}
Let $G$ be a group, and let $a\in G$ with $o(a)=p$, where $p$ is a prime number.
Let $\Xi$ be the connected component of $\Delta(G)$ that contains $a$.
If $\centralizer{G}{a}$ is a $p$-group, then $a\approx b$ for each $b\in\Xi\setminus\{a\}$.
In particular, the only elements of order $p$ connected to $a$ belong to the set $\langle a\rangle\setminus\id$.
\end{lem}

\begin{proof}
Suppose that the result is false.
Thus, there exist elements in $\Xi$ that are \textit{not} adjacent to $a$.
In particular, there must exist a vertex $c\in\Xi$ with $d(a,c)=2$.
Let $a\approx b\approx c$ be a path of length $2$.
Since $b\in\centralizer{G}{a}$, we see that $b$ is a $p$-element.
Hence $|\langle b^{t}\rangle|=p$ for some integer $t$.
The subgroup $\langle a,b\rangle$ is cyclic and, therefore, has a unique subgroup of order $p$, forcing $\langle a\rangle=\langle b^{t}\rangle$.
Now, $a$ is a power of $b$, and so $$\langle a,c\rangle\le\langle b,c\rangle,$$ which yields that $a\approx c$, a contradiction.

Finally, let $b\in\Xi$ with $o(b)=p$.
Then $b\approx a$, and so $\langle a,b\rangle$ has a unique subgroup of order $p$.
Hence $\langle a\rangle=\langle b\rangle$, and it follows that $b\in\langle a\rangle\setminus\id$.
\end{proof}

We now record a basic fact about Frobenius groups, which is a consequence of Theorem 4.1.8 in \cite{kurz}.
We will use this fact in Theorem \ref{completecount}.

\begin{lem}\label{frobcompcontain}
Let $G$ be a Frobenius group with Frobenius kernel $K$ and Frobenius complement $H$.
If $U\le G$ with $U\cap K=\id$, then $U\le H^{g}$ for some $g\in G$.
\end{lem}

At this point, we are able to count the number of connected components of $\Delta(G)$ for a Frobenius group $G$.

\begin{thrm}\label{Frob grp}
Let $G$ be a Frobenius group with Frobenius kernel $K$.
If $K$ is a $p$-group for some prime number $p$, then $\Delta(G)$ has $|K|+m_{p}(K)$ connected components.
If $K$ is not a group of prime power order, then $\Delta(G)$ has $|K|+1$ connected components.
\end{thrm}

\begin{proof}[Sketch of proof]
Let $H$ be a Frobenius complement of $G$.
Each conjugate of $H^{\#}$ constitutes a connected component of $\Delta(G)$, and there are $|K|$ such conjugates.
Since $$K^{\#}=G\setminus\bigcup_{g\in G}H^{g},$$ the conclusion follows from the nilpotency of $K$, Lemma \ref{cyclicgraphofapgroup}, and Lemma \ref{nilpotentnotapgroup}.
\end{proof}

We next present a few basic facts about $2$-Frobenius groups.  We first show that the elements adjacent in $\Delta (G)$ to elements in $D \setminus K$ must lie in $D$.

\begin{lem}\label{swimmingintoD}
If $G$ is a $2$-Frobenius group and $g\in G^{\#}$ with $g\approx d$ for some $d\in D\setminus K$, then $g\in D$.
\end{lem}

\begin{proof}
Since $Kd\in (D/K)^{\#}$, $Kg\in\centralizer{G/K}{Kd}\le D/K$.
Hence $g\in D$.
\end{proof}

Recall that $H$ is a Frobenius complement of $L$.  Similar to the Frobenius group case, we see that the elements in $H^\#$ make up a connected component of $\Delta (G)$ when $G$ is a $2$-Frobenius group.

\begin{lem}\label{Hconnected}
If $G$ is a $2$-Frobenius group, then $H^{\#}$ is a connected component of $\Delta(G)$.
\end{lem}

\begin{proof}
The cyclic graph of a Frobenius complement is connected.  (See the proof of Theorem \ref{Frob grp}.  In fact, the diameter of the cyclic graph of a Frobenius complement is at most $2$.)
Let $h\in H^{\#}$, and suppose that $g\in G^{\#}$ with $g\approx h$.
Write $g=ak$, where $a\in N$, $k\in K$.
Observe that $h\in H\cap H^{ak}=H\cap H^{k}$ since $h=h^{g}=h^{ak}$.
Now $H=H^{k}$, and so $k\in K\cap H=\id$.
Hence $g=a\in\centralizer{N}{h}\le H$. 
\end{proof}

We now show that an element of a $2$-Frobenius group $G$ of prime order that does not lie in any conjugate of $H$ must centralize a nontrivial element of $K$.

\begin{lem}\label{thomp}
Let $G$ be a $2$-Frobenius group, and let $x\in G\setminus\bigcup_{g\in G} H^{g}$.
If $o(x)=p$, a prime, then $\centralizer{K}{x}\neq\id$.
\end{lem}

\begin{proof}
Conjugation by $x$ induces an automorphism of $L$.
If $\centralizer{L}{x}=\id$, then $L$ admits a fixed-point-free automorphism of prime order and is therefore nilpotent by Theorem 10.2.1 in \cite{gorenstein}, a contradiction.
Hence $\centralizer{L}{x}\neq\id$.
By Lemma \ref{Hconnected}, $x$ must centralize an element in the set $L\setminus\bigcup_{g\in L} H^{g}=K^{\#}$.
\end{proof}

Let $G$ be a $2$-Frobenius group.
At this juncture, we are able to obtain a count on the number of connected components of $\Gamma(G)$.
As mentioned, Lemma \ref{Hconnected} actually shows that each $(H^{g})^{\#}$ ($g\in G$) constitutes a connected component of $\Gamma(G)$, and there are $|K|$ such conjugates.

As the subgroup $K$ is nilpotent, $\centerof{K}\ne\id$.
Fix $z\in\centerof{K}^{\#}$.
Now, let $g\in G\setminus\bigcup_{g\in G} H^{g}$ be arbitrary.
For some natural number $t$, the element $x=g^{t}$ has prime order.
(Of course, $x$ will lie outside $\bigcup_{g\in G} H^{g}$, too.)
Hence $\centralizer{K}{x}\ne\id$.
Let $1\ne k\in\centralizer{K}{x}$, and observe that $g\sim x\sim k\sim z$.
It follows immediately that $G\setminus\bigcup_{g\in G} H^{g}$ is a connected component of $\Gamma(G)$.
Hence, if $G$ is a $2$-Frobenius group, then $\Gamma(G)$ has $|K|+1$ connected components.

\section{Main Results}

We now present our main results.
If $G$ is a $2$-Frobenius group and the subgroup $K$ is not a group of prime power order, then a ``nice'' count on the number of connected components of $\Delta(G)$ is available; in fact, under this assumption, $\Delta(G)$ will have the same number of connected components as $\Gamma(G)$.  We now prove Theorem \ref{main1}.


\begin{proof} [Proof of Theorem \ref{main1}]
By Lemma \ref{Hconnected}, the set $\left(\bigcup_{g\in G} H^{g}\right)^{\#}$ is partitioned into $|K|$ connected components of $\Delta(G)$.
	
Let $d\in G\setminus\left(\bigcup_{g\in G} H^{g}\right)$.
Our task is to show that there exists a path from $d$ into $K^{\#}$.
Raise $d$ to an appropriate power to obtain an element $x$ of prime order $p$.
If $x\in K$, then $d\approx x$ is a path from $d$ into $K^{\#}$.
So, assume that $x\notin K$.
Let $Q\in\syl{q}{K}$, where $q\neq p$, and let $K_{0}$ be the normal $q$-complement of $K$.
Write $\overline{G}=G/K_{0}$, and note that $\overline{G}$ is a $2$-Frobenius group.
By Lemma \ref{thomp}, $\overline{x}$ centralizes an element $\overline{y}\in \left(\overline{K}\right)^{\#}$.
Hence $[x,y]\in K_{0}$.
The coset representative $y$ can be chosen to belong to $Q$.
Note that $[x,y]\in Q$ as $Q\nsub G$.
Hence $[x,y]\in Q\cap K_{0}=\id$.
Because $x$ and $y$ commute and have coprime orders, $x\approx y$.
Now, $d\approx x\approx y$.
As $|K|$ is divisible by at least two distinct primes, the set $G\setminus\left(\bigcup_{g\in H} H^{g}\right)$ constitutes a connected component of $\Delta(G)$.
\end{proof}

We next consider the case where $D$ is a $p$-group for some prime $p$.  Notice that $D$ is a $p$-group for some prime $p$ if and only if $K$ and $G/L$ are $p$-groups.  Thus, this next theorem is Theorem \ref{main2}.

\begin{thrm}\label{Dpgroup}
If $G$ is a $2$-Frobenius group and $D$ is a $p$-group for some prime $p$, then $\Delta(G)$ has $|K|+ m_{p}(G)$ connected components.
\end{thrm}

\begin{proof}
The set $\left(\bigcup_{g\in G} H^{g}\right)^{\#}$ is partitioned into $|K|$ connected components of $\Delta(G)$ by Lemma \ref{Hconnected}.
Note that $$G\setminus\bigcup_{g\in G} H^{g}=\left(\bigcup_{g\in G} D^{g}\right)^{\#}.$$
The hypothesis and Lemma \ref{DisHall} yield that $D\in\syl{p}{G}$.
Hence, every subgroup of order $p$ in $G$ is contained in some conjugate of $D$.
Lemma \ref{swimmingintoD} ensures that no point in the set $\left(G\setminus\left(\bigcup_{g\in G} D^{g}\right)\right)^{\#}$ is adjacent to a point in any conjugate of $D$.

So, we are left to show that if $\langle x\rangle$ and $\langle y\rangle$ are distinct subgroups of $G$ of order $p$, then $x$ and $y$ belong to distinct connected components of $\Delta(G)$.
For a contradiction, suppose that $x$ and $y$ lie in the same connected component of $\Delta(G)$, say $\Xi$.
Since $\Xi\subseteq\bigcup_{g\in G} D^{g}$, every element of $\Xi$ is a $p$-element.
Now, write $d(x,y)=n$, and note that $n>1$.
Let $$x\approx x_{1}\approx x_{2}\approx\dots\approx x_{n-1}\approx y$$ be a path of length $n$.
The subgroup $\langle x,x_{1}\rangle$ has a unique subgroup of order $p$.
Since $x_{1}$ is a $p$-element, $o(x_{1}^{t})=p$ for some positive integer $t$.
Thus, $\langle x\rangle=\langle x_{1}^{t}\rangle$.
Similarly, the subgroup $\langle x_{1},x_{2}\rangle$ has a unique subgroup of order $p$.
As $x_{2}$ is a $p$-element, $o(x_{2}^{s})=p$ for some positive integer $s$.
Hence $\langle x_{1}^{t}\rangle=\langle x_{2}^{s}\rangle$.
But now $\langle x\rangle=\langle x_{2}^{s}\rangle$, and so $x=x_{2}^{u}$ for some positive integer $u$.
It follows that $x\approx x_{2}$, which implies $d(x,y)\le n-1$, a contradiction.
\end{proof}

Finally, we have the case that $K$ is a $p$-group and $D$ is not a $p$-group for some prime $p$.  Since this is equivalent to $K$ being a $p$-group and $G/L$ not being a $p$-group, this next theorem is Theorem \ref{main3}.

\begin{thrm}\label{completecount}
Let $G$ be a $2$-Frobenius group, and let $p$ be a prime.
Assume that $K$ is a $p$-group and that $D$ is not a $p$-group.
Then, the number of connected components of $\Delta(G)$ is $$|K|+|H|+m_{p}^{\ast},$$ where $m_{p}^{\ast}$ is the number of subgroups of order $p$ in $G$ that are not centralized by an element of prime order other than $p$.
\end{thrm}

\begin{proof}
As usual, the set $\left( \bigcup_{g\in G} H^{g}\right)^{\#}$ is partitioned into $|K|$ connected components of $\Delta(G)$.  (See Lemma \ref{Hconnected}.)

We claim that every nonidentity element in $\bigcup_{g\in G} D^{g}$ is connected to an element of order $p$ in $G$.
Let $1\neq d\in\bigcup_{g\in G} D^{g}$.
For some positive integer $t$, the element $d^{t}$ has prime order.
If $o(d^{t})=p$, then $d\approx d^{t}$, and we are done.
Assume that $o(d^{t})\in\mathbb{P}\setminus\{p\}$.
Conjugation by $d^{t}$ induces an automorphism of $L$.
By Lemma \ref{thomp}, $\centralizer{K}{d^{t}}\neq\id$.
Let $k\in\centralizer{K}{d^{t}}$ be an element of order $p$.
Since $d^{t}$ and $k$ are commuting elements with coprimes orders, $d^{t}\approx k$.
Hence $d\approx d^{t}\approx k$.
The claim has been established.

Let $a\in G$ be an element of order $p$.
Notice that $\centralizer{G}{a}\cap L=\centralizer{L}{a}\le K$, and so $\centralizer{G}{a}K\cap L=K\centralizer{L}{a}=K$.
In particular, $\centralizer{G}{a}K/K$ intersects $L/K$ trivially.
By Lemma \ref{frobcompcontain}, $\centralizer{G}{a}K/K$ is contained in some conjugate of $D/K$, and, consequently, $\centralizer{G}{a}$ is contained in some conjugate of $D$.

Suppose that $\langle a\rangle$ is a subgroup of $G$ with $|\langle a\rangle|=p$ and that $\langle a\rangle$ is \textit{not} centralized by an element of prime order other than $p$.
Hence $\centralizer{G}{\langle a\rangle}=\centralizer{G}{a}$ is a $p$-group.
Lemma \ref{pgroupcentralizer} yields that the only elements of order $p$ adjacent to $a$ belong to the set $\langle a\rangle\setminus\id$.
We therefore count $m_{p}^{\ast}$ connected components of this type.

Next, let $\pi=\pi(\normalizer{D}{H})\setminus\{p\}$.
Because $\normalizer{D}{H}$ is cyclic, $\normalizer{D}{H}$ has a subgroup $Q$ with $|Q|=|\normalizer{D}{H}|_{\pi}$, the $\pi$-part of $|\normalizer{D}{H}|$.
Note that $|D|_{\pi}=|\normalizer{D}{H}|_{\pi}$, and so $Q$ is a Hall $\pi$-subgroup of $D$.
The solvability of $D$ ensures that $Q^{D}$ is the set of all Hall $\pi$-subgroups of $D$.
Hence, every $p$-regular element of $D$ belongs to a conjugate of $Q$.

Let $\mathcal{D}=\{x\in D\mid o(x)\in\pi\}$, and let $\mathcal{E}$ be the set of all elements of order $p$ that belong to $\centralizer{G}{d}$ for some $d\in\mathcal{D}$.
We shall show that all of the elements in $\mathcal{D}\cup\mathcal{E}$ lie in a single connected component of $\Delta(G)$.

Fix $x\in\mathcal{D}$ with $o(x)=q\in\pi$.
Since $x$ is a $p$-regular element of $D$, the element $x$ lies in some $D$-conjugate of $Q$.
Consequently, $x$ lies in some $D$-conjugate of $\normalizer{D}{H}$, and so $x$ normalizes $H^{d}$ for some $d\in D$.
The subgroup $K$ is nilpotent, and so $\centerof{K}\neq \id$.
As $\centerof{K}\characteristic K\nsub G$, the subgroup $\centerof{K}$ is normal in $G$.
Hence $\centerof{K}H^{d}\langle x\rangle$ forms a subgroup of $G$ and is, in fact, a $2$-Frobenius group.
Applying Lemma \ref{thomp}, $\centralizer{\centerof{K}}{x}\neq \id$.

Let $y\in\mathcal{D}$ with $o(y)=q$.
As $G/L\cong\normalizer{D}{H}$, the factor group $G/L$ is cyclic.
In particular, $G/L$ has a unique subgroup of order $q$.
It follows that $L\langle x\rangle=L\langle y\rangle$.
Now, using Dedekind's Lemma, observe that $$K\langle x\rangle=(D\cap L)\langle x\rangle=D\cap L\langle x\rangle=D\cap L\langle y\rangle=(D\cap L)\langle y\rangle=K\langle y\rangle.$$
Notice that $\id<\centralizer{\centerof{K}}{x}\le \centerof{K\langle x\rangle}\cap K$.
The elements $x$ and $y$ are therefore adjacent to every element in the set $(\centerof{K\langle x\rangle}\cap K)^{\#}$.
Of course, $x$ is adjacent to every element in $\centralizer{K}{x}^{\#}$ and $y$ is adjacent to every element in $\centralizer{K}{y}^{\#}$.
We point out also that this argument shows that if $d\in\mathcal{D}$, then $d$ connects to every element in $d^{D}$.

Now, let $v\in\mathcal{D}$ with $o(v)=r\in\pi\setminus\{q\}$.
Some $D$-conjugate of $v$ lies in $Q$, say $v^{d}$, where $d\in D$.
Now, $$v\approx\dots\approx v^{d}\approx x,$$ as $Q$ is cyclic.
It follows that the set $\mathcal{D}\cup\mathcal{E}$ is contained in a single connected component.

Let $D$ and $D^{g}$, ($g\in G$), be distinct conjugates.
Pick $d_{1}\in D\setminus K$ and $d_{2}\in D^{g}\setminus K$.
The factor group $G/K$ is a Frobenius group and the subgroup $K\langle d_{1},d_{2}\rangle$ cannot be contained in a conjugate of $D/K$.
Hence $K\langle d_{1}, d_{2}\rangle\cap L/K$ is nontrivial.
Working back in $G$, it follows that $K\langle d_{1},d_{2}\rangle \cap L> K$.
Now, the subgroup $\langle d_{1}, d_{2}\rangle$ must contain a nontrivial element from a conjugate of $H$, say $h$.
So $\centralizer{K}{d_{1}}\cap\centralizer{K}{d_{2}}\le\centralizer{K}{h}=\id$.
We conclude that any element of order $p$ in $K$ cannot be adjacent to elements in $D\setminus K$ and $D^{g}\setminus K$ for distinct conjugates $D$ and $D^{g}$.

In light of the previous paragraph, if $g\in G\setminus D$, the sets $\mathcal{D}\cup\mathcal{E}$ and $(\mathcal{D}\cup\mathcal{E})^{g}$ lie in distinct connected components of $\Delta(G)$.
There are $|D^{G}|=|H|$ connected components of this type.
All elements in $G$ have been accounted for, and thus there are $|K|+|H|+m_{p}^{\ast}$ connected components of $\Delta(G)$.
\end{proof}

\section{Examples} \label{last sec}

In this section, we present some examples.
The notation from the previous sections remains in effect.
We shall use $Z_{n}$ to denote the cyclic group of order $n$.  For most of these examples, the computations were done using the computer algebra system Magma.  (See \cite{magma} for information about Magma.)

\begin{enumerate}[leftmargin=*]
\smallskip

\item For this first example, we present a $2$-Frobenius group where the subgroup $D$ is a group of prime power order.  This gives an illustration of computing the formula in Theorem B.
Take $G = S_4$.
Notice that $D$ is a $2$-group, and so by Theorem B, the number of connected components in $\Delta(G)$ is $|K| + m_{2}(G)$.
Observe that $|K| = 4$ and $m_{2}(G)=9$.
Hence, the number of connected components in $\Delta(G)$ is $4 +9 = 13$.
\smallskip

\item Next, we present an example where $K$ is a $p$-group and $D/K$ has order $q \ne p$ and $m_p^*$ is not $0$.  This illustrates one of the possibilities for the formula in Theorem C.  
Consider the $2$-Frobenius group $G=((Z_5\times Z_5)\rtimes Z_3)\times Z_2$.
This group satisfies the hypotheses of Theorem C, and using Theorem C, we see that the number of connected components of $\Gamma(G)$ is $|K| + |H| + m_5^{*}$.
Note that $|K| = 25$ and $|H| = 3$. 
Using Magma, we find $m_5^* = 3$.
We compute that the number of connected components of $\Delta(G)$ is $25+3+3=31$.
In Magma, the group $G$ is \texttt{PrimitiveGroup(25,3)} from the Primitive groups database.  (See \cite{prim} for information regarding the Primitive groups database.)
\smallskip

\item We now consider an example where $K$ is not a group of prime power order.  This demonstrates the formula in Theorem A.  Let $G$ be the $2$-Frobenius group $(((Z_2 \times Z_2) \times (Z_5 \times Z_5))\rtimes Z_3)\rtimes Z_2$.
By Theorem A, the number of connected components of $\Delta(G)$ is $|K| + 1 = 100 + 1 = 101$.   In Magma, this group $G$ can be found as \texttt{SmallGroup (600,179)} from the Small groups database.  (A nice article about the small groups database is \cite{small}.)
\smallskip

\item We present another example where $K$ is a $p$-group and $D/K$ has order $q \ne p$ for a prime $q$, and in this case, however, $m_p^* = 0$.  This give a second example a group meeting the hypotheses of Theorem C.  The purpose of this example is to show that $m_p^* = 0$ can occur.
Take $G=((Z_2 \times Z_2 \times Z_2) \rtimes Z_7)\rtimes Z_3$.
Then $G$ satisfies the hypotheses of Theorem C.
Hence, the number of connected components of $\Gamma (G)$ is $|K| + |H| + m_2^*$.
Note that $|K| = 8$ and $|H| = 7$.
Using Magma, we compute that $m_2^* = 0$.
Thus, the number of connected components of $\Delta (G)$ is $8 + 7 + 0 = 15$.
In Magma, $G$ is \texttt{PrimitiveGroup(8,2)}.
\smallskip

\item We now present a situation where $K$ is a $p$-group, $D$ is not a $p$-group, but $p$ divides $|D:K|$.
Take $G = ((Z_2^{10}) \rtimes Z_{11}) \rtimes Z_{10}$, which is \texttt{PrimitiveGroup(1024,8)} in Magma.
This provides an example where $G$ satisfies the hypotheses of Theorem C and $|K|$ and $|G:L|$ are not coprime.
Hence, the number of connected components of $\Gamma (G)$ is $|K| + |H| + m_2^*$.
Observe that $|K| = 2^{10} = 1024$ and $|H| = 11$.
Using Magma, we compute that $m_2^* = 990$.
We deduce that the number connected components of $\Delta (G)$ is $1024 + 11 + 990 = 2025$.
\smallskip

\item We conclude with one more example of what can happen in Theorem C.  In this example, the subgroup $K$ is a $2$-group and $D/K$ is a $2'$-group so that $|D:K|$ is not a prime.
Take $G = ((Z_2^{15}) \rtimes Z_{151}) \rtimes Z_{15}$.
The group $G$ is too large to appear in Magma's databases, and so we will do the computations explicitly.  
Observe that we can think of $G$ as the additive group of a field of order $2^{15}$ being acted on by the subgroup of order $151$ in the multiplicative group and then all of this being acted on by the Galois group of the field which has order $15$.
It is easy to see that $|K| = 2^{15} = 32768$ and $|H| = 151$.  

We need to compute $m_2^*$.
Suppose that $x\in G$ has order $15$.
It is not difficult to see that $|\centralizer{K}{x}|=2$ and so $\normalizer{G}{\langle x \rangle}=\centralizer{G}{x}= Z_2 \times Z_{15}$.
Next, suppose that $x \in G$ has order $3$, then $|\centralizer{K}{x}| = 2^5 = 32$ and so $\normalizer{G}{\langle x \rangle} = \centralizer{G}{x} = (Z_2^5 \times Z_3) \rtimes Z_5$.  Similarly, if $x \in G$ has order $5$, then  $|\centralizer{K}{x}| = 2^3 = 8$ and $\normalizer{G}{\langle x \rangle} = \centralizer{G}{x} = (Z_2^3 \times Z_5) \rtimes Z_3$.  On the other hand, if $x \in G$ has order $2$, then $\centralizer{G}{x}$ is one of the following: $K$, $K \rtimes Z_3$, $K \rtimes Z_5$, or $K \rtimes Z_{15}$.    

We see that $G$ has $151 \cdot 2^{14}$ subgroups of order $15$.  Each subgroup of order $15$ centralizes one element of order $2$.  On the other hand, each element $x$ of order $2$ that is fixed by an element of order $15$ will fixed by all the of the subgroups of order $15$ in $\centralizer{G}{x}$, and we see that there are $2^{14}$ such subgroups.  Hence, there are $151$ elements of order $2$ in $G$ that are fixed by an element of order $15$.

Observe that $G$ has $151 \cdot 2^{10}$ subgroups of order $3$.
Each subgroup $C$ of order $3$ centralizes $31$ elements of order $2$.
On the other hand, each element $x$ of order $2$ that is fixed by $C$ will be fixed by all of the subgroups of order $3$ in $\centralizer{G}{x}$, and we see that there are $2^{10}$ such subgroups.
Notice that a Sylow $5$-subgroup of $\normalizer{G}{C}$ will centralize one of the elements of order $2$ in $K$, and the Sylow $2$-subgroup of $\centerof{\centralizer{G}{C}}$ will be generated by this element.
Hence, it will be the only element in $\centralizer{K}{C}$ that is fixed by an element of order $15$.
It follows that $C$ centralizes $30$ elements of order $2$ that are not centralized by an element of order $15$, and so $G$ has $151 \cdot 30$ elements of order $2$ that are centralized by an element of order $3$ and not an element of order $15$.

We now do a similar computation to compute the number of elements of order $2$ that are fixed by a subgroup of order $5$ and not a subgroup of order $15$.
Observe that $G$ has $151 \cdot 2^{12}$ subgroups of order $5$.
Each subgroup $C$ of order $5$ centralizes $7$ elements of order $2$.
On the other hand, each element $x$ of order $2$ that is fixed by $C$ will be fixed by all of the subgroups of order $5$ in $\centralizer{G}{x}$, and we see that there are $2^{12}$ such subgroups.
Notice that a Sylow $5$-subgroup of $\normalizer{G}{C}$ will centralize one of the elements of order $2$ in $K$, and the Sylow $2$-subgroup of $\centerof{\centralizer{G}{C}}$ will be generated by this element.
Hence, it will be the only element in $\centralizer{K}{C}$ that is fixed by an element of order $15$.
It follows that $C$ centralizes $6$ elements of order $2$ that are not centralized by an element of order $15$, and so, $G$ has $151 \cdot 6$ elements of order $2$ that are centralized by an element of order $5$ and not an element of order $15$.

The number of elements of order $2$ in $G$ that are fixed by an element of prime order other than $2$ is $151 + 151 \cdot 30 + 151 \cdot 6 = 151 \cdot 37$.
On the other hand, $G$ contains $2^{15} - 1 = 217 \cdot 151$ elements of order $2$.
This implies that $m_{2}^{*} = 217 \cdot 151 - 37 \cdot 151 = 180 * 151 = 27180$.
Hence $\Delta(G)$ has $32768 + 151 + 27180 = 59919$ connected components.

\end{enumerate}

\end{document}